\documentclass[11pt]{elsarticle}
\usepackage{amsmath, amssymb}
\usepackage{tikz}
\usepackage{tikz-3dplot}

\usetikzlibrary{arrows.meta, positioning, 3d}

\usepackage{caption}
\usepackage{amsmath,amssymb,amsfonts}
\usepackage{algorithmicx}
\usepackage{graphicx}
\usepackage{textcomp}
\usepackage{xcolor}
\usepackage{lineno}
\usepackage{float}
\usepackage{resizegather}
\usepackage[ruled,linesnumbered]{algorithm2e}
\usepackage{algpseudocode}
\usepackage{epsfig}
\usepackage[all]{xy}

\usepackage{tikz} 
\usetikzlibrary{decorations.pathreplacing,angles,quotes}
\usetikzlibrary{matrix,arrows,decorations.pathmorphing}
\usepackage{tikz-cd}
\usetikzlibrary{arrows}
\usetikzlibrary{decorations.markings}
\definecolor{darkgreen}{rgb}{0.0,0.5,0.0}

\usepackage{amsmath,amssymb,amsthm}
\usepackage{url}
\newtheorem{defi}{Definition}
\newtheorem{prop}{Proposition}

\newtheorem{theo}{Theorem}
\newtheorem{cor}{Corollary}

\newtheorem{exa}{Example}
\usepackage{subcaption}

\newcommand{\pt}{\text{ }\forall\text{ }}
\newcommand{\tq}{\text{ }:\text{ }}
\newcommand{\N}{\mathbb{N}}

\newcommand{\io}{[0,1]}
\newcommand{\Ln}{L_n}
\newcommand{\ci}{\mathbb{I}}

\begin{document}

\begin{frontmatter}

\title{ On the simplicial structure of uncertain information}
\author[uja]{Juan Martínez-Moreno}
\ead{jmmoreno@ujaen.es}
\author[linares]{Diego García-Zamora\corref{mycorrespondingauthor}}
\cortext[mycorrespondingauthor]{Corresponding author}
\ead{dgzamora@ujaen.es}

\address[uja]{Department of Mathematics, Universidad de Jaén, Campus Las Lagunillas, 23071, Jaén, Spain}
\address[linares]{Department of Mathematics, Universidad de Jaén, Campus Científico Tecnológico de Linares, 23700, Linares, Spain}
\begin{abstract}
The mathematical representation of uncertainty has led to a proliferation of preference structures, such as interval-valued fuzzy sets, intuitionistic fuzzy sets, and various granular models. While these extensions are often studied independently, they share profound geometric and topological foundations. This paper provides a unifying framework by identifying these disparate structures with the simplicial geometry of $n$-dimensional fuzzy sets. We first conduct an extensive revision of both classical and modern preference structures, demonstrating that they are distinct semantic interpretations of the same underlying topological objects within the lattice $L_n$. Building on this unification, we introduce a new, highly interpretable preference structure based on Deck-of-Cards membership functions. This approach generalizes the revised models by providing a flexible mechanism to represent complex membership degrees through monotonic sequences. Furthermore, we establish a formal simplicial structure for the set of multidimensional fuzzy sets $L_\infty$. By employing face and degeneracy maps, we demonstrate how this framework unifies existing models into a single simplicial set, allowing for the consistent transformation of information across different levels of granularity. The examples provided illustrate the utility of this simplicial connection in several contexts, offering a robust topological foundation for future developments in fuzzy set theory.
\end{abstract}

\begin{keyword}
Fuzzy sets extensions, Multidimensional fuzzy sets, Simplicial sets, Uncertainty modeling, Granular computing
\end{keyword}

\end{frontmatter}
\section{Introduction}

The mathematical modeling of uncertain information has undergone significant evolution, moving from classical fuzzy sets to more sophisticated structures capable of representing complex cognitive nuances. A foundational advancement in this direction is the development of $n$-dimensional fuzzy sets \cite{bedregal2011, multidimensional, SYL2010}. These sets represent membership values as $n$-tuples of real numbers within the unit interval $[0,1]$ arranged in a non-decreasing order, forming the lattice $L_n$. Geometrically, this lattice is identified with the order polytope $\mathcal{O}(P_n)$ associated with a linear chain \cite{stanley1986}.

Despite the robustness of $n$-dimensional fuzzy sets, the extant literature has seen a proliferation of diverse preference structures, often introduced as distinct extensions. For instance, interval-valued and Atanassov intuitionistic fuzzy sets \cite{ATANASSOV1989343, ATANASSOV_interval, CousoAtanasov} provide a range for pessimistic and optimistic evaluations. Similarly, more recent models such as Basic Uncertain Information (BUI) granules \cite{BUI}, Cognitive Interval Information (CII) granules \cite{HCYL}, and Asymmetric Interval Numbers (AIN) \cite{Salabun} have been proposed to handle specific types of granular uncertainty. Other structures, including fuzzy rough sets \cite{pawlak1982rough, YAO1998227}, shadowed sets \cite{Shadowed}, grey sets \cite{YANG2012249}, and vague sets \cite{vague}, also aim to capture different facets of vagueness. However, these theories are frequently studied in isolation from algebraic or logical perspectives \cite{DESCHRIJVER2003227}, leaving a research gap in the form of a unified geometric and topological framework that connects them.

In this paper, we address this gap by framing these various isomorphisms geometrically, identifying these structures as distinct semantic interpretations of the same topological objects: the $n$-dimensional order polytopes. Our contribution is threefold. First, we provide an extensive revision of classical and contemporary preference structures (such as BUI, CII, AIN, and Picture Fuzzy Sets \cite{Cuong2014}), demonstrating their equivalence to specific $L_n$ of different dimensions. Second, we introduce a new interpretable preference structure that generalizes all the revised approaches whose semantic basis can be constructed using ased on Deck-of-Cards membership functions \cite{DoCMF}. Finally, we establish a simplicial structure for the set of multidimensional fuzzy sets $L_\infty$. By leveraging the formal language of simplicial sets \cite{curtis, munkres} and complexes \cite{kozlov}, we provide a unifying structure that connects these existing models through face and degeneracy maps.

The remainder of this manuscript is organized as follows. Section \ref{sec:prelim} establishes the preliminary concepts regarding multidimensional fuzzy sets and simplicial topology. Section \ref{sec:revision} details the equivalences between various preference structures in the literature and $n$-dimensional fuzzy sets. Section \ref{sec:n-ICUI} presents our new proposal based on Deck-of-Cards membership functions. Section \ref{sec:simplex} describes the simplicial structure of $L_\infty$ and provides examples of how this framework unifies and connects different granularities of information. Finally, the conclusions of the study are summarized in Section \ref{sec:conclusion}.

\section{Preliminary concepts}\label{sec:prelim}
This section introduces the main concepts required to understand the proposal.

\subsection{Multidimensional fuzzy sets}

The concept of $n$-dimensional fuzzy sets extends fuzzy sets to include, among others, interval-valued fuzzy sets and intuitionistic fuzzy sets. In these sets, membership values are represented as $n$-tuples of real numbers within the unit interval $[0,1]$, forming $n$-dimensional intervals arranged in nondecreasing order. The essence of $n$-dimensional fuzzy sets lies in accommodating multiple levels of uncertainty in membership degrees. To do so, Bedregal et al.  \cite{bedregal2011} considered the lattices
\begin{equation*}
L_n=L_n([0,1]):=\{(x_1,..., x_{n})\in [0,1]^n: \,x_1 \le x_2\le ... \le x_n\},
\end{equation*}
for each $n\in\mathbb{N}=\{1,2,3,...\}$, which was called an upper simplex and whose elements are called $n$-dimensional intervals. We will denote by $\pi_i (x)\in [0,1]$ the projection of $x$ over its $i$-th coordinate.

Note that $L_1=[0,1]$, whereas $L_2$ reduces to the lattice of all the closed subintervals of the unit interval $[0, 1]$. The natural partial order on $\Ln$ is given by 
\begin{equation*}
    (x_1,...,x_n)\le^n (y_1,...,y_n)\Leftrightarrow x_i\le y_i, \text{ for all } i=1,...,n.
\end{equation*}
From \cite{SYL2010} we know that $(\Ln,\le^n)$ is a complete lattice, with its greatest and least elements being $(1,...,1)$ and $(0,..., 0)$,  respectively.

If we denote by $[n] = \{1, \dots, n\}$ the linear chain of $n$ elements, then the set $\Ln$ coincides exactly with the set of order-preserving maps from $[n]$ to the unit interval $[0, 1]$. In the language of combinatorial geometry, this identifies $\Ln$ as the order polytope $\mathcal{O}(P_n)$ associated with the chain $P_n$, $1<2\cdots < n$, \cite{stanley1986}.

This geometric perspective provides a structural insight: the $n$-dimensional volume (in Lebesgue's sense) of $\Ln$ is exactly $1/n!$. Compared to the unit volume of the hypercube $[0,1]^n$, this reduction quantifies the information gain obtained by imposing the monotonicity constraint inherent in cognitive uncertainty. Furthermore, the vertices of this polytope correspond to the characteristic functions of the filters of the poset, representing the logical transitions from absolute false to absolute true.

Moreover, we will consider the strict version of $\Ln$ defined by 
\begin{equation*}
L_n^*([0,1)):=\{(x_1,..., x_{n})\in [0,1]^n: \,x_1 < x_2< ... < x_n\}.
\end{equation*}
Any element of $L_n^*([0,1))$ is called a strict $n$-dimensional interval.  There is no doubt that $L_n^*([0,1))\subset \Ln$.

While $\Ln$ is defined by linear inequalities,  it is instructive to analyze its vertex representation.  As the order polytope of a linear chain, $\Ln$ is the convex hull of $n+1$ specific vertices in $\mathbb{R}^n$.
Let $v_k \in \{0,1\}^n$ be the characteristic vectors representing the filters of the chain $[n]$. These vertices are defined as:
$    v_0 = (0, 0, \dots, 0, 0)$, $
    v_1 = (0, 0, \dots, 0, 1)$,
$    v_2 = (0, 0, \dots, 1, 1), $ ...,
$    v_n = (1, 1, \dots, 1, 1).$ 
Geometrically, $\Ln$ is the convex hull of these points:
\begin{equation*}
    \Ln = \text{conv}\{v_0, v_1, \dots, v_n\} = \left\{ \sum_{i=0}^n \lambda_i v_i : \sum_{i=0}^n \lambda_i = 1, \lambda_i \ge 0 \right\}.
\end{equation*}
Since the set of vectors $\{v_0, \dots, v_n\}$ is affinely independent, $\Ln$ constitutes an $n$-dimensional simplex embedded within the unit hypercube $[0,1]^n$. 

Consequently, the strict set $L_n^*([0,1))$ (the non-degenerate elements) corresponds to the relative interior of this simplex, characterized by strict positivity of the barycentric coordinates ($\lambda_i > 0$ for all $i$). Figure \ref{fig:order_polytope} illustrates this construction for dimensions $n=2$ and $n=3$.

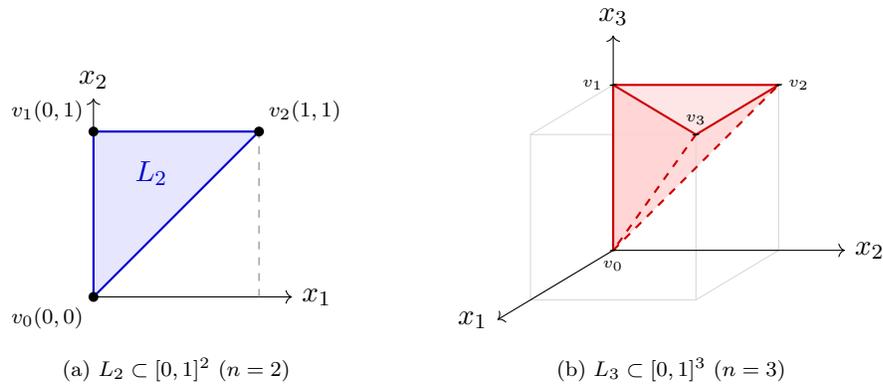
\begin{figure}[htbp]
    \centering
    \begin{subfigure}[b]{0.48\textwidth}
        \centering
        \begin{tikzpicture}[scale=2.2, line join=round]
            \draw[->] (0,0) -- (1.2,0) node[right] {$x_1$};
            \draw[->] (0,0) -- (0,1.2) node[above] {$x_2$};
            \draw[gray, dashed] (1,0) -- (1,1) -- (0,1);
            \fill[blue!10] (0,0) -- (0,1) -- (1,1) -- cycle;
            \draw[thick, blue!80!black] (0,0) -- (0,1) -- (1,1) -- cycle;
            \fill (0,0) circle (0.03); \node[below left, font=\scriptsize] at (0,0) {$v_0(0,0)$};
            \fill (0,1) circle (0.03); \node[above left, font=\scriptsize] at (0,1) {$v_1(0,1)$};
            \fill (1,1) circle (0.03); \node[above right, font=\scriptsize] at (1,1) {$v_2(1,1)$};
            \node[blue!80!black] at (0.35, 0.75) {$L_2$};
        \end{tikzpicture}
        \caption{$L_2 \subset [0,1]^2$ ($n=2$)}
        \label{fig:order_polytope_2d}
    \end{subfigure}
    \hfill 
    \begin{subfigure}[b]{0.48\textwidth}
        \centering
        \begin{tikzpicture}[scale=2.2, line join=round, x={(-0.5cm,-0.3cm)}, y={(1cm,0cm)}, z={(0cm,1cm)}]
            \coordinate (O) at (0,0,0); \coordinate (X) at (1,0,0);
            \coordinate (Y) at (0,1,0); \coordinate (Z) at (0,0,1);
            \coordinate (XY) at (1,1,0); \coordinate (XZ) at (1,0,1);
            \coordinate (YZ) at (0,1,1); \coordinate (XYZ) at (1,1,1);
            \draw[gray!30] (O) -- (X) -- (XY) -- (Y) -- cycle;
            \draw[gray!30] (Z) -- (XZ) -- (XYZ) -- (YZ) -- cycle;
            \draw[gray!30] (O) -- (Z); \draw[gray!30] (X) -- (XZ);
            \draw[gray!30] (Y) -- (YZ); \draw[gray!30] (XY) -- (XYZ);
            \draw[->] (O) -- (1.4,0,0) node[left] {$x_1$};
            \draw[->] (O) -- (0,1.4,0) node[right] {$x_2$};
            \draw[->] (O) -- (0,0,1.3) node[above] {$x_3$};
            \coordinate (v0) at (0,0,0); \coordinate (v1) at (0,0,1);
            \coordinate (v2) at (0,1,1); \coordinate (v3) at (1,1,1);
            \fill[red!10, opacity=0.7] (v0) -- (v1) -- (v3) -- cycle; 
            \fill[red!20, opacity=0.7] (v0) -- (v1) -- (v2) -- cycle; 
            \fill[red!5, opacity=0.5] (v1) -- (v2) -- (v3) -- cycle;
            \draw[thick, red!80!black] (v0) -- (v1);
            \draw[thick, red!80!black] (v1) -- (v2);
            \draw[thick, red!80!black] (v2) -- (v3);
            \draw[thick, red!80!black, dashed] (v3) -- (v0); 
            \draw[thick, red!80!black, dashed] (v0) -- (v2);
            \draw[thick, red!80!black] (v1) -- (v3);
            \fill (v0) circle (0.02); \node[below, font=\tiny] at (v0) {$v_0$};
            \fill (v1) circle (0.02); \node[left, font=\tiny] at (v1) {$v_1$};
            \fill (v2) circle (0.02); \node[right, font=\tiny] at (v2) {$v_2$};
            \fill (v3) circle (0.02); \node[above, font=\tiny] at (v3) {$v_3$};
        \end{tikzpicture}
        \caption{$L_3 \subset [0,1]^3$ ($n=3$)}
        \label{fig:order_polytope_3d}
    \end{subfigure}

    \caption{Visualization of the Order Polytope for $n=2$ and $n=3$, showing the simplex boundaries within the unit hypercube.}
    \label{fig:order_polytope}
\end{figure}

Using the polytope $\Ln$ as a basis, \cite{SYL2010} introduced the notion of $n$-dimensional fuzzy sets on a non-empty universe of discourse $X$.

\begin{defi}[\cite{SYL2010}]
Let $n\in\mathbb{N}$. An $n$-dimensional
fuzzy set $A$ over $X$ is a map $A:X\to \Ln$ given by
\begin{equation*}
    A(x) = (A_1 (x), . . . ,{A_n} (x)),
\end{equation*}
with $A_1 (x)\leq. . . \leq A_n (x)$, for every $x\in X$. In addition, each  fuzzy set ${A_i}:X\to [0,1]$ is called the $i$-th projection of $A$. 
\end{defi}

Since $L_1= [0,1]$, a 1-dimensional fuzzy set is just a classical fuzzy set.  Similarly, a 2-dimensional fuzzy set is equivalent to an interval fuzzy set (also equivalent to an intuitionistic fuzzy set), since
$L_2$ reduces to the usual lattice $\mathbb{I}$ of all the closed subintervals of the unit interval $[0,1]$. Extending this idea, in \cite{multidimensional}, the authors considered the set
\begin{equation*}
    {L}_\infty={L}_\infty([0,1]):=\bigcup_{n\ge 1} \Ln
\end{equation*}
in which every element can be seen as an $n$-dimensional interval in $\Ln$, for some $n\in \mathbb{N}$. Using this set as a codomain, we obtain multidimensional fuzzy sets on the non-empty universe of discourse $X$.

\begin{defi}[\cite{multidimensional}] A multidimensional fuzzy set $A$ over $X$ is a pair of maps $(p,A):X\to \mathbb{N}\times \mathcal{L}_\infty([0,1))$ such that
$$A(x) = {( {A_1} (x), . . . , {A_{p(x)}} (x)) \in L_{p(x)}([0,1]),}$$
for every $x\in X$.
\end{defi}

\subsection{Simplicial sets}

A simplicial complex (simplicial set) is a fundamental concept in topology \cite{curtis, munkres} that has gained significant importance in the study of data. Unlike graph models, which primarily focus on pair-wise interactions, simplicial complexes offer a richer framework for characterizing complex connections within or between datasets \cite{kozlov}. They can be constructed computationally from various data types, including point clouds, matrices, volumetric functions, networks, and graphs, making them essential tools in data science and computer science.

A simplicial set  $X_\bullet$ can be thought of as a way to represent a space by piecing together $n$-simplices in a combinatorial manner. Geometrically, a 0-simplex represents a point, a 1-simplex represents a line segment, a 2-simplex represents a triangle, a 3-simplex represents a tetrahedron, and so forth. An n-simplex can be conceptualized as a polyhedron formed by the convex hull of $n+1$ geometrically independent points (i.e., they do not lie in any hyperplane of dimension $n$) within the Euclidean space $\mathbb{R}^n$. Formally, 

\begin{defi}
A simplicial set $X_\bullet$ consists of a sequence of sets $\{X_n\},$ for $n\ge 0$, together with the face maps $d_i:X_n\to X_{n-1}$, $0\leq i\leq n$, and the degeneracy maps $s_j:X_n \to X_{n+1}$, $0\leq j\leq n$ satisfying the following {\it simplicial identities}:
\begin{equation*}
\begin{aligned}
d_i d_j &= d_{j-1} d_i  \;\;\;\; \text{ if } i<j \\
s_i s_j &= s_j s_{i-1} \vspace{2cm} \;\;\;\;\text{ if } i>j \\
d_i s_j &= \left\lbrace 
\begin{array}{ll}
s_{j-1} d_i & \text{ if } i<j \\
id  & \text{ if } i=j,\,j+1 \\
s_jd_{i-1} & \text{ if } i>j+1.
\end{array}
\right. 
\end{aligned}
\end{equation*}
If all $X_i$ are lattices and $d_i$ and $s_j$ are lattice homomorphisms, the simplicial set $X_\bullet$ is a simplicial lattice. Elements of the set $X_n$ are called $n$-simplices.
\end{defi}

Often, a simplicial set is written as
\[ \begin{tikzcd}
\cdots X_3  \rar[shift left=1.5] \rar[shift left=4.5] \rar[shift right=1.5] \rar[shift right=4.5]  & 
X_2
\rar[shift left=3,"d_2"] \rar[shift right=3,"d_0"'] \rar
\lar[shift left=3] \lar[shift right=3] \lar & 
X_1 \rar[shift left=2, "d_1"] \rar[shift right=2, "d_0"']  \lar[shift left=1.5] \lar[shift right=1.5] 
&
X_0. \lar["s_0" description]
\end{tikzcd} \]

In algebraic topology, the topological $n$-simplex $\Delta_n$ consists of the points $(w_0,w_1,\cdots,w_n)$ in $[0,1]^{n+1}$ such that $\sum_{i=0}^n w_i=1$. Such a simplex (usually related to weighting vectors) has $(n+1)$-faces. Intuitively, a simplicial set is a collection of  $n$-simplices representing $\Delta_n$  together with face maps, indicating how to glue simplices, and degeneracies.

In the case that we have only face maps satisfying $d_i d_j = d_{j-1} d_i$,  for  $i<j$, we call $X_\bullet$ a semisimplicial set (i.e., without degeneracy maps).
Semi-simplicial sets are simply a generalisation of simplicial sets.  

\section{Some preference structures that are equivalent to $n$-dimensional fuzzy sets}\label{sec:revision}

This section continues the work in \cite{CousoAtanasov} to show that many of the preference structures in the extant literature are nothing but $n$-dimensional fuzzy sets. For instance, it is well-known that fuzzy sets are $1$-dimensional fuzzy sets. In the same way, we present other equivalences between other recently-proposed structures and the classical $n$-dimensional fuzzy sets.

\subsection{2-dimensional simplices: $L_2$.}

First, let us recall that the set $\ci$ of all intervals of the form $[a,b]\subseteq [0,1]$ is equivalent to $L_2$. 
Intervals are utilized for modeling many decision-making problems, keeping in mind that the length $b-a$ means the uncertainty. The interval $[a,b]$ gives the range between the pessimistic and the optimistic evaluation.

Given a universe of discourse $X$ (we can think of alternatives for a decision problem), let us denote the set of maps from $X$ to the order polytope $L_2$ by $\mathcal{F}_{L_2}(X)$. It is clear that $\mathcal{F}_{L_2}(X)$, i.e., the set of $2$-dimensional fuzzy sets, is indeed the set of the interval-valued fuzzy sets, which is equivalent to the  Atanassov intuitionistic fuzzy sets via the isomorphism between $L_2$ and the set $\mathbb{A}_2=\{(a,b)\in [0,1]^2: a+b\le 1\}$.

\begin{prop} The following sets are isomorphic lattices.
a) the set $L_2$; b) $\mathbb{I}$ and c) $\mathbb{A}_2$. 
\end{prop}
\begin{proof}
Given $(x_1,x_2)\in L_2$, $0\le x_1\le x_2\le 1$. Since  $0 \le 1-x_2\le 1-x_1 $, we get that
$0\le x_1+(1-x_2)\le x_1+(1-x_1)=1$, i.e. $(x_1,1-x_2)\in \mathbb{A}_2$. Conversely, given $(a,b)\in \mathbb{A}_2$, $a+b\le 1$ and $a\le 1-b$, i.e., $(a,1-b)\in L_2$.
\end{proof}

This issue has been studied in the literature \cite{ATANASSOV1989343,CousoAtanasov}. However, interval-valued and intuitionistic fuzzy sets are not the only preference structures that are isomorphic. Below, we provide an extensive compendium of other examples.

\begin{defi} \cite{BUI} A basic uncertain information (BUI) granule is a pair $(x,c)\in [0,1]\times [0,1]$  in which $x$ is the evaluation value and $c$ is the certainty degree of $x$.
\end{defi}

We will denote by $BUI$ the set of all BUI granules and by $BUI^*$ the subset of $BUI$ with elements $(x,c)$ such that $c>0$.

\begin{prop}\label{prop:BUI_Atan} The following sets are bijective:
a) the set $BUI^*$; b)  the set $L_2\backslash \{(0,1)\} $; c) $\mathbb{I}\backslash \{[0,1]\} $ and d) $\mathbb{A}_2\backslash \{(0,0)\} $. 
\end{prop}

\begin{proof}
We can  transform  a BUI granule $(x,c)$ into an interval $I_{(x,c)}\in L_2$ by the certainty/uncertainty dilatation:
$$I_{(x,c)}=[x-(1-c)x,x+(1-c)(1-x)]=[cx,cx+1-c].$$
Since $c\neq 0$,  $(cx,cx+1-c)\in L_2\backslash \{(0,1)\} $.
Conversely, given $(x_1,x_2)\in L_2$ with $(x_1,x_2)\neq (0,1)$, we can obtain a BUI granule $(x,c)$ with $c=x_1+1-x_2\in (0,1]$ and $x=\frac{x_1}{x_1+1-x_2}\in [0,1]$. 
\end{proof}

From the previous result, one could induce a lattice structure on the set of BUI granules:
\begin{gather*}
    (x_1,c_1)\le_{BUI} (x_2,c_2)\iff I_{(x_1,c_1)}\le I_{(x_2,c_2)} \\\iff c_1x_1\le c_2x_2\text{ and }c_1(x_1-1)\le c_2(1-c_2).
\end{gather*}
However, in our opinion, the partial order obtained is not entirely natural. In addition, note that other preference structures based on BUI are also within the scope of this discussion \cite{STANDR-BUI}.

Additionally, the set-theoretic equivalence between axiomatic fuzzy rough sets (pairs of lower/upper approximations) and interval-valued Fuzzy sets is well-documented in the literature \cite{DESCHRIJVER2003227,YAO1998227}. However, these equivalences have traditionally been studied from algebraic (lattice-theoretic) or logical perspectives. On the contrary, the goal of our contribution is to frame these isomorphisms geometrically: we identify all these structures as distinct semantic interpretations of the same topological object: the 2-dimensional order polytope $L_2$.

\begin{defi}[Fuzzy Rough Set \cite{pawlak1982rough}]
Let $X$ be a universe of discourse. A Fuzzy Rough Set $\mathcal{A}$ in $X$ is characterized by a pair of fuzzy sets, the lower approximation $ {\underline{A}}$ and the upper approximation $ {\overline{A}}$, satisfying the condition:
\begin{equation*}
    {\underline{A}}(x) \le  {\overline{A}}(x), \quad \forall x \in X
\end{equation*}
where $ {\underline{A}},  {\overline{A}}: X \to [0,1]$. The pair $( {\underline{A}}(x),  {\overline{A}}(x))$ describes the granular uncertainty of an element $x\in X$.
\end{defi}
Of course, we obtain the following result.
\begin{theo}[Rough Set Representation]
The space of all fuzzy rough sets on a universe $X$ is isomorphic to the space of 2-dimensional fuzzy sets $\mathcal{F}_{L_2}(X)$.
\end{theo}
\begin{proof}
Let us recall that $L_2$ is defined as the simplex $\{(x_1, x_2) \in [0,1]^2 : x_1 \le x_2\}$.
We construct the isomorphism $\Phi$ by assigning each fuzzy rough set $\mathcal{A} = ( {\underline{A}},  {\overline{A}})$, the element of $\mathcal{F}_{L_2}(X)$ determined by 
\begin{equation*}
    \Phi(\mathcal{A})(x) = ( {\underline{A}}(x),  {\overline{A}}(x))\pt x\in X.
\end{equation*}
By the definition of rough approximations, $ {\underline{A}}(x) \le  {\overline{A}}(x)$, which satisfies exactly the defining condition of $L_2$. Thus, $\Phi(\mathcal{A})(x)$ lies in $L_2$ and the mapping is well-defined. In addition, the mapping $\Phi$ is trivially bijective. Any pair $(u, v) \in L_2$ defines a valid fuzzy rough membership at a point, where $u$ is the necessary membership (lower) and $v$ is the possible membership (upper). In addition, note that the standard inclusion order for rough sets is given by $\mathcal{A} \subseteq \mathcal{B} \iff  {\underline{A}} \le  {\underline{B}} \text{ and }  {\overline{A}} \le  {\overline{B}}$. This corresponds exactly to the component-wise partial order defined on $L_2$.
\end{proof}

This bijection allows us to interpret concepts from granular computing geometrically within the simplex $L_2$. From this point of view: (i) Classical sets correspond to the diagonal $D = \{(x,x): x \in \{0,1\}\} \subset L_2( [0,1])$; (ii) {Extreme Rough Sets}, also knwon in the literature as Shadowed sets \cite{Shadowed}, correspond to maps into the vertices of $L_2$; and (iii) the measure of ``roughness'' of a set, typically defined as $ {\overline{A}}(x) -  {\underline{A}}(x)$, corresponds precisely to the {length} of the interval represented by the point in $L_2$, or its distance from the main diagonal (certainty). Thus, $L_2$ constitutes the {topological skeleton} of rough set theory.

Following a similar discussion, we can assess the following result.
\begin{theo}[Structural Equivalence of $L_2$-Fuzzy Sets]\label{L2isom}
    The following theories are equivalent to $L_2$-fuzzy sets: (i) 2-dimensional fuzzy sets \cite{bedregal2011},  (ii) intuitionistic fuzzy sets \cite{ATANASSOV1989343}, (iii) interval type 2 fuzzy sets \cite{CousoAtanasov}; (iv) grey sets \cite{YANG2012249}; (v) fuzzy rough sets \cite{pawlak1982rough}; and (vi) vague sets \cite{vague}.
\end{theo}
Note that also BUI granules may be embedded within this structure, although they are not bijective \cite{BUI}.

\subsection{3-dimensional simplices: $L_3$.}
Now, let us identify some of the structures of the literature with  $L_3$ fuzzy sets. Given a universe of discourse $X$, let us denote the set of maps from $X$ to the order polytope $L_3$ by $\mathcal{F}_{L_3}(X)$. Let us consider the following structures:

\begin{defi} \cite{HCYL} A cognitive interval information (CII) granule is a the pair $(x,[a^-,a^+])\in [0,1]\times \mathbb{I}$  such that 
$x\in [a^-,a^+]$ is the evaluation value and $[a^-,a^+]$ is called the acceptance interval of $x$. 
\end{defi}
 
\begin{defi}[\cite{Salabun}] An Asymmetric Interval Number (AIN) is a pair represented as $[a^-,a^+]_x$  such that $x\in [a^-,a^+]$ is the expected value and $[a^-,a^+]$ is an interval contained in $[0,1]$.
\end{defi}

Additionally, picture fuzzy sets \cite{Cuong2014} are based on the lattice
$$\mathbb{A}_3=\{(a_1,a_2,a_3)\in \mathbb{I}^3: a_1+a_2+a_3\le 1\}.$$

\begin{prop}\label{prop:L3isom} The following sets are bijective: i) the set $L_3$; ii) the set of CII granules; iii) the set of AINs, and iv) the set $\mathbb{A}_3$. 
\end{prop}
\begin{proof}
Since $x\in [a^-,a^+]$, then $a^-\le x\le a^+$ and $(a^-,x,a^+)\in L_3$. 
\end{proof}
Note that Pythagorean \cite{Pythagorean} fuzzy sets are also bijective to this kind of structure.

Since both a CII granule and an element in $L_3$ are uniquely determined by an ordered triad, it is easy to define a complete lattice structure on the first set that is isomorphic to the second; similarly for the AINs set. Inspired by \cite{Salabun}, we can define an asymmetry coefficient $A$ for the values $(x_1,x_2,x_3)\in L_3$ as 
\begin{equation*}
A_{(x_1,x_2,x_3)}=
\begin{cases}
\frac{x_1+x_3-2x_2}{x_1+x_3},& \text{ if } x_1\neq x_3\\
0,& \text{ if } x_1=x_2=x_3.
\end{cases}
\end{equation*}
 Let us summarize the previous results.
 \begin{theo}[Structural Equivalence of $L_3$-Fuzzy Sets]
The following preference structures are mathematically equivalent to the lattice of 3-dimensional fuzzy sets $\mathcal{F}_{L_3}(X)$: (i) 3-dimensional fuzzy sets \cite{SYL2010}, (ii) CII granules on a set $X$ \cite{HCYL}, (iii) AINS on a set $X$ \cite{Salabun}, (iv) Picture Fuzzy Sets \cite{Cuong2014}, and (v) Triangular Fuzzy Sets \cite{METRIC}.
\end{theo}

\subsection{4, 5, and 6-dimensional simplices: $L_4, L_5,$ and $L_6$.}
To illustrate the use of $L_4$ in the literature, let us start by considering interval-valued intuitionistic fuzzy sets \cite{ATANASSOV_interval}, which are based on the set 
\begin{equation*}
    \mathbb{A}_4=\{([a_1,a_2],[b_1,b_2])\in \mathbb{I}^2: a_2+b_2\le 1\}
\end{equation*}

\begin{prop}There is a bijection between the sets  $L_4$ and $\mathbb{A}_4$.
\end{prop}
\begin{proof}
From a pair of intervals $([a_1,a_2],[b_1,b_2])\in \mathbb{A}_4$, $a_1\le a_2\le 1-b_2\le 1-b_1$, i.e.,
$(a_1, a_2, 1-b_2, 1-b_1)\in L_4$. Conversely, if $(x_1,x_2,x_3,x_4)\in L_4$, then $([x_1,x_2],[1-x_4,1-x_3])\in \mathbb{A}_4$, since $x_2-x_3\le 0$.
\end{proof}

On the other hand, shadowed sets \cite{Shadowed} offer a mechanism to simplify fuzzy information by preserving only the essential structures: the core (certainty), the excluded area (impossibility), and the shadow (uncertainty). We can demonstrate that the membership structure of a shadowed set is canonically isomorphic to the order polytope $L_4$.

\begin{defi}[Shadowed Set \cite{Shadowed}]
Let $X$ be a universe of discourse. A Shadowed Set $\mathcal{S}$ on $X$ is defined as a mapping $S: X \to \{0, 1, [0,1]\}$.
\end{defi}
However, in its continuous extensions (often used in optimization and granular computing), a shadowed set is characterized by two nested intervals that delineate the `shadow region'. These intervals are indeed the core and the support of a certain corresponding fuzzy set. 
\begin{theo}[Shadowed Set Isomorphism]
The space of shadowed sets defined on $[0,1]$ is isomorphic to the simplicial lattice $L_4$.
\end{theo}

\begin{proof}
The mapping is constructed naturally. For any element $x \in X$, the state of the shadowed set is fully described by the pair of intervals $(C(x),S(x))$ corresponding to the core $C(x)=[C^-(x),C^+(x)]$ and the support $S(x)=[S^-(x),S^+(x)]$. Since $C(x)\subseteq S(x)\subseteq [0,1]$, the mapping valuated on $L_4$
\begin{equation*}
    (C(x),S(x))\to (S^-(x),C^-(x),C^+(x),S^+(x)])
\end{equation*}
is a bijection.
\end{proof}

Similar to CII,  in \cite{HCYL2}, the Interval type Cognitive Interval information granule (ICII) is presented as a pair of nested intervals: the evaluation and the acceptance intervals. This uncertainty allows other  values  to
be accepted or tolerated by the decision-maker, providing a convenient
and effective intersubjective communication and negotiation
between her/him and other experts.

\begin{defi} \cite{HCYL2}  An interval (type) cognitive interval information (ICII) granule is a pair of the form $([x^-,x^+],[a^-,a^+])\in \mathbb{I}\times \mathbb{I}$  such that 
$[x^-,x^+]\subset [a^-,a^+]$ is the interval evaluation value and $[a^-,a^+]$ is called the acceptance interval.
\end{defi}
Of course, every CII granule $(x,[a^-,a^+])$ is an ICII granule, taking as interval $[x^-,x^+]=[x,x]$. The following result is also clear. 

\begin{prop}
There is a bijection between the set of ICII granules and the set $L_4$. 
\end{prop}

\begin{defi}[\cite{RBUI}] A relative basic
uncertain information (RBUI) granule is expressed by a triple $(x,[a^-,a^+],c)\in [0,1]\times \mathbb{I}\times [0,1]$, where $x\in [a^-,a^+]$ is the plausible value,  $[a^-,a^+]$ is a refined interval in which the true value is known to be
included, $c$ is the relative certainty degree of $x$.
\end{defi}

\begin{prop}
    The set of RBUI granules can be embedded into $L_4$.
\end{prop}
\begin{proof}
We can  transform  a RBUI granule $(x,[a^-,a^+],c)$ into 
\begin{equation*}
    (a^-,xc+(1-c)a^-,xc+(1-c)a^+,a^+)\in L_4.
\end{equation*}
\end{proof}

It is easy to see that the set of CII can be isomorphic to the space of all triangular fuzzy numbers, but with very different meanings.

\begin{defi} \cite{Jin2023} A granule of interval type
Basic uncertain information (ItBUI) is a pair with the form $(x, [c^-, c^+])\in [0,1]\times \mathbb{I}$, in which $x\in [0,1]$ is the evaluation value and $[c^-, c^+]$ is the interval certainty degree of a, measuring the plausibility of being trusted, convincing or believable of input value $x$.
\end{defi}
\begin{prop}
  The set of ItBUI can be embedded into $L_4$.
\end{prop}
\begin{proof}
Let us transform $(x, [c^-, c^+])$ into intervals $I_{(x,c)}\in \mathbb{I}$ by the certainty/uncertainty dilation:
$$I_{(x,c^-)}=[x-(1-c^-)x,x+(1-c^-)(1-x)]=[c^-x,c^-x+1-c^-].$$
$$I_{(x,c^+)}=[x-(1-c^+)x,x+(1-c^+)(1-x)]=[c^+x,c^+x+1-c^+].$$
Since $c^-\le c^+$,  $I_{(x,c^+)}\subset I_{(x,c^-)}$ and $(c^-x, c^+x, c^+x+1-c^+,c^-x+1-c^-)\in L_4$.
\end{proof}

\begin{defi} \cite{Jin2023}  A granule of BUI type Basic Uncertain information (BtBUI) is a pair of the form $(y,(x,c))$, in which $y \in [0, 1]$ is the
evaluation value and $(x,c)$ is the BUI certainty degree measuring the quantity of being trusted, convincing, or believable, of input value $y$.
\end{defi}

\begin{prop}
    The set of BtBUIs can be embedded into $L_4$.
\end{prop}
\begin{proof}
    By utilizing the certainty/uncertainty dilation, we map the BUI certainty $(x,c)$ degree into $[c^-,c^+]$. Consequently, the pair $(y,[c^-,c^+])$ is an ItBUI and can be remapped into $L_4$.    
\end{proof}

Let us summarize in a single result this discussion.

\begin{theo}[Structural Equivalence of $L_4$-Fuzzy Sets]
The space of 4-dimensional fuzzy sets $\mathcal{F}_{L_4}(X)$ is isomorphic to the following uncertainty modeling frameworks: (i) Interval-valued Intuitionistic Fuzzy Sets \cite{ATANASSOV_interval}, (ii) Shadowed Sets \cite{Shadowed}, (iii) ICII granules on a set $X$  \cite{HCYL2}, (iv) RBUI granules on a set $X$ \cite{RBUI}, (v) ItBUI granules on a set $X$ \cite{Jin2023}, and Trapezoidal Fuzzy Sets on $X$ \cite{METRIC}.
\end{theo}

There are also examples in the literature of preference structures that can be seen as 5-dimensional simplices.

\begin{defi} \cite{HCYL} 
A cognitive uncertain information (CUI) granule is a triad with the form  $(x,[a_1,a_2],[u_1,u_2])\in [0,1]\times \mathbb{I}\times \mathbb{I}$  such that  $x\in [a_1,a_2]$ and $[a_1,a_2]\subseteq [u_1,u_2]$. $x$ is the evaluation value,  $[a_1,a_2]$  is called the acceptance interval and  $[0,1]\setminus [u_1,u_2]$ is called the unaccepted area.
\end{defi}

\begin{prop}There is a bijection between the set of CUIs and $L_5$. 
\end{prop}
\begin{proof}
Since $x\in [a_1,a_2]$ and $[a_1,a_2]\subseteq [u_1,u_2]$,  then $u_1\le a_1\le x\le a_2\le u_2$ and $(u_1,a_1,x,a_2,u_2)\in L_5$. 
\end{proof}

We conclude the section by showing one preference structure that is bijective with $L_6$.

\begin{defi} \cite{HCYL2} 
An interval (type) cognitive uncertain information (ICUI) granule is a triad with the form  $([x_1,x_2],[a_1,a_2],[u_1,u_2])\in \mathbb{I}\times \mathbb{I}\times \mathbb{I}$  such that  $ [x_1,x_2]\subset [a_1,a_2]\subset [u_1,u_2]$. $[x_1,x_2]$ is the evaluation interval value,  $[a_1,a_2]$  is called the acceptance interval and  $[0,1]\setminus [u_1,u_2]$ is called the unaccepted area.
\end{defi}

\begin{prop}
There is a bijection between the set of ICUI granules and $L_6$. 
\end{prop}

Every CUI granule $(x,[a_1,a_2],[u_1,u_2])$  is an ICUI granule, taking as evaluation interval $[x^-,x^+]=[x,x]$.

\subsection{$n$-dimensional simplices: $\Ln$.}
Here, we cover other structures that require a fixed but arbitrary number of values $n\in\N$.
\subsubsection{$n$-dimensional fuzzy sets and weighting vectors}
For the sake of completeness, let us consider 
\begin{equation*}
    \mathbb{A}_n=\left\{(x_1,x_2,...,x_n)\in [0,1]^n:\sum \limits_{i=1}^nx_i\le 1\right\}
\end{equation*}
which is the basis for $n$-intuitionistic fuzzy sets and $n$-dimensional fuzzy sets \cite{SYL2010}. In addition, let us recall the set of weighting vectors
\begin{equation*}
    \Delta_n=\{(w_1,...,w_n)\tq \sum_{i=1}^nw_i=1\}.
\end{equation*}

\begin{theo}
    The following sets are bijective to each other: i) $L_n$, ii) $\mathbb{A}_n$, and iii) $\Delta_{n+1}$.  
\end{theo}
\begin{proof}
    It is clear that $L_n$ is bijective with the set $\Delta_{n+1}$ through the bijection $\phi:L_n\to\Delta_{n+1}$ defined by 
    \begin{equation*}    \phi(x_1,...,x_n)_i=x_i-x_{i-1}\ (i=1,...,n+1)  \pt (x_1,...,x_n)\in L_n
    \end{equation*}
    using the convention $x_{n+1}=1$ and $x_0=0$.

    On the other hand, we can consider the isomorphism $\psi:\mathbb{A}_n\to\Ln$  defined as follows: 
    \begin{gather*}        \psi(x_1,x_2,...,x_n)=\left(x_1,x_1+x_2,x_1+x_2+x_3,...,\sum \limits_{i=1}^nx_i\right) \pt(x_1,x_2,...,x_n)\in \mathbb{A}_n.
    \end{gather*}
\end{proof}
\subsubsection{Hesitant Cognitive Uncertain Information}
    Recently, Jin et al.  \cite{HesitCUI} have introduced the notions of typical hesitant monopolar cognitive interval information and typical hesitant cognitive uncertain information, where they use a finite set (hesitant) as the evaluation value.

    \begin{defi} \cite{HesitCUI}\label{defhesitCUI} Let $k\in\mathbb{N}$.  A (typical) $k$-hesitant monopolar cognitive uncertain information (HMCUI) granule is a pair with the form $(\{x_i\}_{i=1}^k,[a^-,a^+])$$\allowbreak\in L_k\times \mathbb{I}$. That the finite and ordered list $\{x_i\}_{i=1}^k$
    is called the (typical) hesitant evaluation value, $[a^-,a^+]$ is called the acceptance interval, and $a^-\le x_1 <x_2<\cdots <x_{k-1}<x_k\le a^+$.
    \end{defi}

    If $k=1$, then a 1-HMCUI granule reduces to a CII granule. Moreover, we can observe that given a $k$-HMCUI granule $(\{x_i\}_{i=1}^k,[a^-,a^+])$, we have $k$ different HCII granules, given by $(x_i,[a^-,a^+])$, $i=1,...,k$,  within the same acceptance interval and satisfying $x_1 <x_2<\cdots <x_{k-1}<x_k$.

    \begin{prop}There is an injection from the set of $k$-HMCUI granules to the set  $L_{k+2}$. 
    \end{prop}

    \begin{defi} \cite{HesitCUI} Let $k\in\mathbb{N}$.  A (typical) $k$-hesitant cognitive uncertain information (HCUI) granule is a triad with the form $(\{x_i\}_{i=1}^k,[a^-,a^+],[u^-,u^+])$$\allowbreak\in L_k\times \mathbb{I}^2$  such that $\{x_i\}_{i=1}^k$
    is called the (typical) hesitant evaluation value,  $u^-\le a^-\le x_1 <x_2<\cdots <x_{k-1}<x_k\le a^+\le u^+$,  $[a^-,a^+]$ is called the acceptance interval and  $[0,1]\setminus [u^-,u^+]$ is called the unaccepted area.
    \end{defi}

    \begin{prop}There is an injection from the set of $k$-HCUI granules to the set  $L_{k+4}$. 
    \end{prop}
    Note that if $k=1$, then a 1-HCUI granule reduces to a CUI granule. 

\subsubsection{Probabilistic Linguistic Term Sets}

    Linguistic Term Sets (LTSs) are the basis of linguistic decision-making \cite{Herrera00b}. An LTS can be defined as a totally ordered finite list of linguistic labels $S=\{s_0,s_1,...,s_n\}$, where each $s_i$ represents a possible value of a linguistic variable and $s_1\prec s_2\prec ... \prec s_n$, where $\prec$ is a certain total order. One of the structures based on LTSs is the so-called Probabilistic LTS.
    \begin{defi}[PLTS \cite{PLTS}]
    Let $S=\{s_1,...,s_n\}$ be a LTS. A probabilistic linguistic term set on $S$ (PLTS) is a set of pairs 
    \begin{equation*}
        \{(s_i,p_i)\in S\times\io\tq \sum_{i=1}^np_i\leq1\}
    \end{equation*}
    such that each $s\in S$ appear at most once.
    \end{defi}

Clearly, 
\begin{prop} Given an LTS, $S=\{s_1,...,s_{n}\}$, there is a bijection between the set of all the PLTSs on $S$ and the sets $\mathbb{A}_{n}$ and $\Delta_{n+1}$.
\end{prop}
Note that in the literature, the PLTS definition allows that $\sum_{i=1}^np_i<1$, but we could also consider, without restriction, that such a sum is equal to $1$, by including one more dimension to stand for the remaining uncertainty, i.e., $p_{n+1}=1-\sum_{i=1}^np_i$.

\begin{theo}[Structural Equivalence of $L_n$-Fuzzy Sets]
The space of $n$-dimensional fuzzy sets $\mathcal{F}_{L_n}(X)$ is isomorphic to the following uncertainty modeling frameworks: (i) weighting vectors on a set $X$, (ii) $n$-intuitionistic fuzzy sets, (iii) Probabilistic Linguistic Term Sets \cite{PLTS}, (iv) $n$-polygonal fuzzy numbers \cite{polygonal}.

\end{theo}

\section{A new preference structure based on $L_n$}\label{sec:n-ICUI}
Based on the discussion of the previous section, we can see that many of the preference structures in the literature are based on nested intervals (or $L_n$ for some $n\in\N$). In this line, we can generalize all of them into a common structure.

\begin{defi}[$n$-dimensional Interval (type) cognitive uncertain information] Let $S=\{s_1,...,s_n\}$ be a LTS such that $s_1\prec s_2\prec ... \prec s_n$. 
An $n$-dimensional interval (type) cognitive uncertain information (n-ICUI) granule is an n-tuple with the form  $([x_1^-,x_1^+],...,[x_n^-,x_n^+])\in  \mathbb{I}^n$  such that  $[x_n^-,x_n^+]\subseteq...\subseteq [x_2^-,x_2^+] \subseteq[x_1^-,x_1^+]\subseteq [0,1] $. The interval  $[x_i^-,x_i^+]$ is called  interval with  level of acceptance $s_i$, for $i=1,...,n$, whereas $[0,1]\setminus [x_1^-,x_1^+]$ is called the unaccepted area.
\end{defi}
Thus, there is an isomorphism between this structure and our simplicial structure.
\begin{theo}
    The set of n-ICUI granules is bijective with the space $L_{2n}$
\end{theo}
\begin{proof}
Since $[x_n^-,x_n^+] \subseteq ... \subseteq [x_2^-,x_2^+]\subseteq [x_1^-,x_1^+] $ then we can conclude that the sequence $(x_n^-,...,x_1^-,x_1^+,...,,x_n^+)\in L_{2n}$. 
\end{proof}
In the particular case that $[x_n^-,x_n^+]=[x,x]$,  we can introduce the concept of $n$-ICII, which is equivalent to $L_{2n-1}$.  
\begin{cor}
    The set of n-ICUI granules whose $n$-th interval $[x_n^-,x_n^+]$ is degenerated is  bijective with the space $L_{1+2(n-1)}$.
\end{cor}

Note that for an LTS $S_2=\{ \text{Acceptable}, \text{Perfect evaluation value} \}$, a  2-ICUI is just an ICII, whereas a 2-ICUI  with $x_2^-=x_2^+$ is just a CII. For an LTS $S_3=\{  \text{Acceptable}, \text{Almost perfect evaluation value},  \text{Perfect evaluation value}\}$, a 3-ICUI is just an ICUI, and a 3-ICUI with $x_3^-=x_3^+$ is just a CUI.

In general, we can think about an $n$-ICUI  as a preference structure for which an evaluator or expert assesses the evaluation value
(for an object under evaluation) as nested intervals for which he/she has a certain confidence. Therefore, the evaluator provides a sequence of nested intervals  $[x_i^-,x_i^+]$ , $i=1,...,n$ and s/he accepts each value that fallsinto the interval $[x_i^-,x_i^+]$  with  level of acceptance $s_i$. Note that this structure can be elicited using the Deck of Card Membership Function introduced in \cite{DoCMF}. Indeed, this structure generalizes the notion of $n$-polygonal fuzzy numbers \cite{polygonal} as well as the concept of step fuzzy numbers \cite{StepFN}:
\begin{cor}
    Given a finite and ordered set of $\alpha$-levels $A\subset[0,1]$, the set of step fuzzy sets on $A$ and the set of polygonal fuzzy sets are bijective to the set of n-ICUI granules.
\end{cor}

\section{One Simplicial set to unify them all}\label{sec:simplex}
In the previous sections, we have discussed how many classical and new preference structures are mathematically equivalent to $n$-dimensional fuzzy sets, i.e., the set of mappings $X\to L_n$. This section is devoted to unifying all these structures via the notion of a simplicial set.

The idea is to highlight the interrelation between the different levels of $L_n$, looking for an internal simplicial structure.
For this, we consider the face maps $d_i:L_{n}\to L_{n-1}$, $0\le i\le n-1$,
defined by
\begin{equation*}
    d_i(x_1,...,x_n)= \begin{cases}
(x_{2},...,x_n),&i=0,\\
(x_1,...,x_{i},x_{i+2},...,x_n),&0<i<n-1,\\
(x_{1},...,x_{n-1}),&i=n-1
\end{cases}
\end{equation*}
 and the degeneracy maps  $s_j:L_{n}\to L_{n+1}$ given by $\ 0\le j\le n-1$
\begin{equation*}
    s_j(x_1,...,x_n)= 
(x_1,...,, x_j,x_{j+1},x_{j+1},x_{j+2},...,x_n),\ 0\le j\le n-1.
\end{equation*}

\begin{figure}[htbp]
    \centering
    \begin{subfigure}[b]{0.45\textwidth}
        \centering
        \begin{tikzpicture}[scale=3] 
            \draw[->] (-0.1,0) -- (1.2,0) node[right] {$x_1$};
            \draw[->] (0,-0.1) -- (0,1.2) node[above] {$x_2$};
            
            \fill[blue!10, draw=blue, thick] (0,0) -- (0,1) -- (1,1) -- cycle;
            \node[blue] at (0.15,0.85) {$L_2$};
            
            \filldraw (0,0) circle (0.5pt) node[anchor=north east] {$v_0(0,0)$};
            \filldraw (0,1) circle (0.5pt) node[anchor=east] {$v_1(0,1)$};
            \filldraw (1,1) circle (0.5pt) node[anchor=west] {$v_2(1,1)$};
            
            \coordinate (X) at (0.3, 0.7);
            \filldraw[black] (X) circle (0.7pt) node[anchor=south west] {$x(x_1, x_2)$};
            
            \draw[dashed, ->, red, thick] (X) -- (0.3, 0) node[below] {$d_1(x) = x_1$};
            \draw[dashed, ->, green!50!black, thick] (X) -- (0, 0.7) node[left] {$d_0(x) = x_2$};
        \end{tikzpicture}
        \caption{\small Representation of $d_0, d_1$ in $L_2$.}
        \label{fig:l2_repro}
    \end{subfigure}
    \hfill
    \begin{subfigure}[b]{0.45\textwidth}
        \centering
        \begin{tikzpicture}[scale=2.6] 
            \begin{scope}[shift={(0,0)}, x={(-0.5cm,-0.3cm)}, y={(1cm,0cm)}, z={(0cm,1cm)}]
                
                \coordinate (O) at (0,0,0); \coordinate (X) at (1,0,0);
                \coordinate (Y) at (0,1,0); \coordinate (Z) at (0,0,1);
                \coordinate (XY) at (1,1,0); \coordinate (XZ) at (1,0,1);
                \coordinate (YZ) at (0,1,1); \coordinate (XYZ) at (1,1,1);
            
                \draw[gray!20] (O) -- (X) -- (XY) -- (Y) -- cycle;
                \draw[gray!20] (Z) -- (XZ) -- (XYZ) -- (YZ) -- cycle;
                \draw[gray!20] (O) -- (Z); \draw[gray!20] (X) -- (XZ);
                \draw[gray!20] (Y) -- (YZ); \draw[gray!20] (XY) -- (XYZ);
                
                \draw[->, thick] (O) -- (1.3,0,0) node[left] {$x_1$};
                \draw[->, thick] (O) -- (0,1.3,0) node[right] {$x_2$};
                \draw[->, thick] (O) -- (0,0,1.3) node[above] {$x_3$};
            
                \coordinate (v0) at (0,0,0);
                \coordinate (v1) at (0,0,1);
                \coordinate (v2) at (0,1,1);
                \coordinate (v3) at (1,1,1);
            
                \fill[red!10, opacity=0.4] (v0) -- (v1) -- (v3) -- cycle; 
                \fill[red!20, opacity=0.4] (v0) -- (v1) -- (v2) -- cycle; 
                \fill[red!5, opacity=0.3] (v1) -- (v2) -- (v3) -- cycle;
            
                \draw[thick, red!80!black] (v0) -- (v1);
                \draw[thick, red!80!black] (v1) -- (v2);
                \draw[thick, red!80!black] (v2) -- (v3);
                \draw[thick, red!80!black, dashed] (v3) -- (v0); 
                \draw[thick, red!80!black, dashed] (v0) -- (v2);
                \draw[thick, red!80!black] (v1) -- (v3);
                
                \coordinate (Xpt) at (0.2, 0.5, 0.9);
                \filldraw[black] (Xpt) circle (0.015) node[anchor=south west, font=\tiny] {$x(x_1, x_2, x_3)$};

                \coordinate (PX) at (0.2, 0.25, 0.8);
                \draw[->, blue, thick, dashed] (Xpt) -- (PX) node[midway, left, font=\tiny] {$d_1$};
                \filldraw[blue] (PX) circle (0.015) node[anchor=north east, font=\tiny] {$d_1(x) \in L_2$};

                \draw[blue!40, very thin] (v0) -- (v1) -- (v3) -- cycle;
                
                \node[below left, font=\tiny] at (v0) {$v_0$};
                \node[left, font=\tiny] at (v1) {$v_1$};
                \node[right, font=\tiny] at (v2) {$v_2$};
                \node[above, font=\tiny] at (v3) {$v_3$};
            \end{scope}
        \end{tikzpicture}
        \caption{\small Representation of $d_1$ in $L_3$.}
        \label{fig:l3_repro}
    \end{subfigure}
    \caption{Visual comparison of projections in $L_2$ and $L_3$.}
\end{figure}
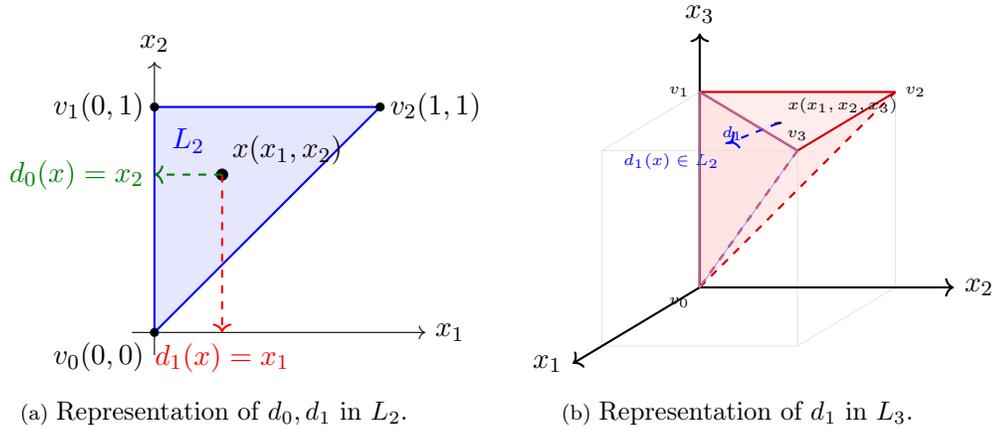

\begin{figure}[htbp]
    \centering
    \begin{subfigure}[b]{0.48\textwidth}
        \centering
        \begin{tikzpicture}[scale=2.8]
            \draw[->] (-0.1,0) -- (1.2,0) node[right] {$x_1$};
            \draw[->] (0,-0.1) -- (0,1.2) node[above] {$x_2$};
            
            \fill[blue!5, draw=blue!30, thin] (0,0) -- (0,1) -- (1,1) -- cycle;
            \node[blue!40] at (0.15,0.85) {$L_2$};
            
            \draw[very thick, black] (0,-0.2) -- (1,-0.2);
            \node[below] at (0.5,-0.2) {$L_1$};

            \filldraw (0,0) circle (0.5pt) node[anchor=north east] {$v_0$};
            \filldraw (0,1) circle (0.5pt) node[anchor=east] {$v_1$};
            \filldraw (1,1) circle (0.5pt) node[anchor=west] {$v_2$};
            
            \coordinate (X1) at (0.6, -0.2);
            \coordinate (SX) at (0.6, 0.6);
            \filldraw[black] (X1) circle (0.7pt) node[below] {$x$};
            
            \draw[red, ultra thick] (0,0) -- (1,1);
            \draw[dashed, ->, red, thick] (X1) -- (0.6, 0.55);
            \filldraw[red] (SX) circle (0.8pt) node[anchor=south west] {   \textbf{ } $s_0(x)$};
            
            \node[red, rotate=45, above] at (0.4,0.4) {\tiny Image of $s_0$};
        \end{tikzpicture}
        \caption{\small Degenaracy $s_0: L_1 \hookrightarrow L_2$ embeds $L_1$ as diagonal.}
        \label{fig:deg_s0_2d}
    \end{subfigure}
    \hfill
    \begin{subfigure}[b]{0.48\textwidth}
        \centering
        \begin{tikzpicture}[scale=2.6] 
            \begin{scope}[shift={(0,0)}, x={(-0.5cm,-0.3cm)}, y={(1cm,0cm)}, z={(0cm,1cm)}]
                \draw[gray!15] (0,0,0) -- (1,0,0) -- (1,1,0) -- (0,1,0) -- cycle;
                \draw[gray!15] (0,0,1) -- (1,0,1) -- (1,1,1) -- (0,1,1) -- cycle;
                \draw[gray!15] (0,0,0) -- (0,0,1); \draw[gray!15] (1,0,0) -- (1,0,1);
                \draw[gray!15] (0,1,0) -- (0,1,1); \draw[gray!15] (1,1,0) -- (1,1,1);
                
                \draw[->, thick] (0,0,0) -- (1.3,0,0) node[left] {$x_1$};
                \draw[->, thick] (0,0,0) -- (0,1.3,0) node[right] {$x_2$};
                \draw[->, thick] (0,0,0) -- (0,0,1.3) node[above] {$x_3$};
            
                \coordinate (v0) at (0,0,0); \coordinate (v1) at (0,0,1);
                \coordinate (v2) at (0,1,1); \coordinate (v3) at (1,1,1);
            
                \draw[thick, red!20] (v0) -- (v1) -- (v2) -- (v3) -- cycle;
                \draw[thick, red!20] (v1) -- (v3);
                
                \fill[blue!30, opacity=0.6] (v0) -- (v1) -- (v3) -- cycle;
                \draw[blue, thick] (v0) -- (v1) -- (v3) -- cycle;
                \node[blue, font=\tiny, rotate=55] at (0.2, 0.2, 0.6) {Im $s_0$};

                \fill[green!50!black!30, opacity=0.4] (v0) -- (v2) -- (v3) -- cycle;
                \draw[green!50!black, thick] (v0) -- (v2) -- (v3) -- cycle;
                \node[green!50!black, font=\tiny, rotate=-15] at (0.4, 0.6, 0.7) {Im $s_1$};
                
                \node[below left, font=\tiny] at (v0) {$v_0$};
                \node[left, font=\tiny] at (v1) {$v_1$};
                \node[right, font=\tiny] at (v2) {$v_2$};
                \node[above, font=\tiny] at (v3) {$v_3$};
            \end{scope}
        \end{tikzpicture}
        \caption{\small Degeneracy faces $s_0, s_1: L_2 \hookrightarrow L_3$.}
        \label{fig:deg_s01_3d}
    \end{subfigure}
    \caption{Visualizing the embedding of lower-dimensional order polytopes via degeneracy maps.}
\end{figure}
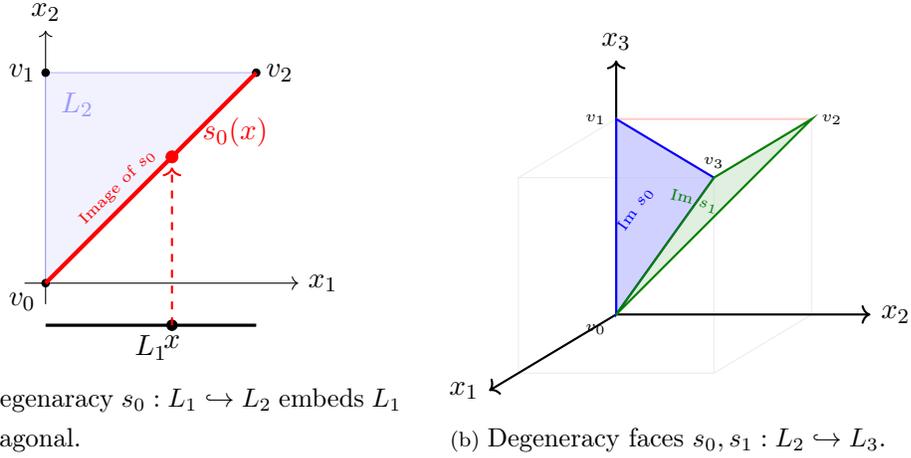

We can represent the structure consisting of $ L_\infty$ together  with these mappings as follows:
\[ \begin{tikzcd}
L_\bullet =&\cdots L_4  \rar[shift left=1.5] \rar[shift left=4.5] \rar[shift right=1.5] \rar[shift right=4.5]  & 
L_3
\rar[shift left=3,"d_2"] \rar[shift right=3,"d_0"'] \rar
\lar[shift left=3] \lar[shift right=3] \lar & 
L_2 \rar[shift left=2, "d_1"] \rar[shift right=2, "d_0"']  \lar[shift left=1.5] \lar[shift right=1.5] 
&
L_1=[0,1]. \lar["s_0" description]
\end{tikzcd} \]
Consequently, we obtain the following result:
\begin{theo}
The sequence of sets $L_\bullet=\{L_n\}_{n\geq0}$ together with th previously defined mappings $d_i:L_{n}\to L_{n-1}$, $0\le i\le n-1$, and  $s_j:L_{n}\to L_{n+1}$, $\ 0\le j\le n-1$, is a simplicial lattice whose $n-$simplices are $L_n$.
\end{theo}

\begin{proof}
It is routine to check simplicial identities and lattice conditions.
\end{proof}
Note that, essentially, the simplicial lattice $L_\bullet$ is related to multi-dimensional fuzzy sets $L_\infty$, but is essentially different. Although they are constructed using the lattices $L_n$, $L_\infty$ lacks the mappings that endow $L_\bullet$ with the simplical lattice structure.

\begin{exa}
    As an example of the potential of this simplicial lattice structure, note that by using the face and degeneracy maps, we can increase or decrease the dimension and, with this, construct an element of $L_2$ from another in $L_3$ and vice versa. For instace, to build mappings from $L_3$ to $L_2$, we can consider the face maps $d_i:L_{3}\to L_{2}$, $0\le i\le 2$ given by
\begin{equation*}
    d_i(x_1,x_2.,x_3)= \begin{cases}
(x_{2},x_3),&i=0,\\
(x_1,x_3),&i=1,\\
(x_{1},x_2),&i=2,
\end{cases}.
\end{equation*}
Using Propositions \ref{prop:BUI_Atan} and \ref{prop:L3isom}, the previous mappings induce $d_i:CII\to BUI$, $0\le i\le 2$, given by 
\begin{equation*}
    d_i(x,[a^-,a^+])= \begin{cases}
\left(\frac{x}{x+1-a^+},x+1-a^+\right),&i=0,\\
\left(\frac{a^-}{a^-+1-a^+},a^-+1-a^+\right),&i=1,\\
\left(\frac{a^-}{a^-+1-x},a^-+1-x\right),,&i=2.
\end{cases}
\end{equation*}
In the same way, we can consider the degeneracy maps  $s_0,s_1:L_{2}\to L_{3}$ defined as
\begin{equation*}
    s_j(x_1,x_2)= \begin{cases}
(x_1,x_1,x_{2}),&j=0,\\
(x_{1},x_{2},x_2),&j=1
\end{cases}
\end{equation*}
to induce degeneracy maps  $s_0,s_1:BUI\to CII$ by 
\begin{equation*}
    s_j(x,c)=s_j(cx,cx+1-c)= \begin{cases}
(cx,[cx,cx+1-c])&j=0,\\
(cx+1-x,[cx,cx+1-c]),&j=1
\end{cases}
\end{equation*}
\end{exa}
\begin{exa}
Let us consider again $n$-ICUI granules defined in the previous section. Let us recall the case in which the interval associated with $s_n$ is degenerated. This scenario corresponds with the degeneracy map $s_n:L_{2n-1}\to L_{2n}$ given by 
\begin{equation*}
    s_n(x_n^-,...,x_2^-,x,x_2^+,...,,x_n^+)=(x_n^-,...,x_2^-,x,x,x_2^+,...,,x_n^+).
\end{equation*}
Keep in mind that the interpretation of $s_n$ in cognitive uncertain language is just an inclusion of the set of $n-$ICII into the set of $n$-ICUI. Roughly speaking, $L_{2n-1}\subset L_{2n}$. In fact, the set $L_{2n}\setminus L_{2n-1}$ reduces to degenerative ones. Other particular cases can be constructed when some $x_i^-=x_{i+1}^-$ or $x_i^+=x_{i+1}^+$, for some $j$ (or even more). All of these correspond with the other degeneracy mappings $s_j$., i.e, in  $L_{2n}\setminus L_{2n-1}$. The simplicial structure of $L_\bullet$ allows us to go from $n-$ICII to $n$-ICUI and back again.
\end{exa}
\begin{exa}
As another example, we highlight that this structure allows looking at a triangular fuzzy number as the natural particular case of a trapezoidal fuzzy number (provided that the supports of both families whithin bounded subset of the real line). Since triangular fuzzy numbers are bijective with $L_3$ and trapezoidal fuzzy numbers are bijective with $L_4$, it suffices to consider the degeneracy map $s_1:L_3\to L_4$ defined by $(a,b,c)\to (a,b,b,c)\pt (a,b,c)\in L_3$. Of course, the face maps allow defuzzifying a trapezoidal fuzzy number by forgetting one component, i.e., by losing information.
\end{exa}

\begin{exa}
The $n$-ICUI structure provides a robust geometric bridge between discrete preference levels and the continuous framework of $\alpha$-cuts in fuzzy set theory. Consider an evaluator providing a 3-ICUI granule, which corresponds to an element in $L_6$ given by the sequence $(x_3^-, x_2^-, x_1^-, x_1^+, x_2^+, x_3^+) \allowbreak\in L_6$. 

Geometrically, this can be interpreted as a "stack" of nested intervals, where each interval $[x_i^-, x_i^+]$ represents the support of the evaluation at a specific confidence or acceptance level $s_i$. This is exactly the constructive definition of a fuzzy set through its $\alpha$-cuts. If we fix a set of levels $0 < \alpha_1 < \alpha_2 < \alpha_3 \le 1$, the $n$-ICUI defines a step fuzzy set $A$ where each interval is the $\alpha_i$-cut, i.e.,  $A_{\alpha_i} = [x_i^-, x_i^+]$.

\begin{figure}[H]
    \centering
    \begin{tikzpicture}[scale=1.2]
        \draw[->] (-0.5,0) -- (5.5,0) node[right] {$X$};
        \draw (0,0.1) -- (0,-0.1) node[below] {$0$};
        \draw (5,0.1) -- (5,-0.1) node[below] {$1$};

        \draw[thick, blue] (1,0.5) -- (4.5,0.5);
        \fill[blue] (1,0.5) circle (1.5pt);
        \fill[blue] (4.5,0.5) circle (1.5pt);
        \node[left, blue, font=\scriptsize] at (1,0.5) {$[x_1^-, x_1^+]$};

        \draw[thick, red] (1.8,1.2) -- (3.8,1.2);
        \fill[red] (1.8,1.2) circle (1.5pt);
        \fill[red] (3.8,1.2) circle (1.5pt);
        \node[left, red, font=\scriptsize] at (1.8,1.2) {$[x_2^-, x_2^+]$};

        \draw[thick, green!60!black] (2.5,1.9) -- (3.2,1.9);
        \fill[green!60!black] (2.5,1.9) circle (1.5pt);
        \fill[green!60!black] (3.2,1.9) circle (1.5pt);
        \node[left, green!60!black, font=\scriptsize] at (2.5,1.9) {$[x_3^-, x_3^+]$};

        \draw[dashed, gray!50] (1,0.5) -- (1,0);
        \draw[dashed, gray!50] (4.5,0.5) -- (4.5,0);
        \draw[dashed, gray!50] (2.5,1.9) -- (2.5,0);
        \draw[dashed, gray!50] (3.2,1.9) -- (3.2,0);
    \end{tikzpicture}
    \caption{Representation of a 3-ICUI as nested intervals.}
\end{figure}
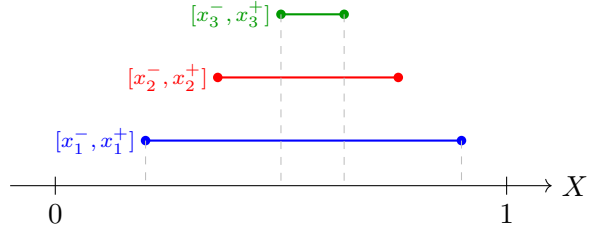

By applying the degeneracy maps $s_j$ within the simplicial lattice $L_\bullet$, we can refine this representation. For instance, $s_j$ acts as an interpolation operator that adds a redundant level, which can then be "pushed" to accommodate a new $\alpha$-cut. Conversely, the face maps $d_i$ allow for the simplification of a fuzzy set into a lower-dimensional preference structure (like a CII or BUI granule) by omitting specific levels of granularity, as illustrated below.

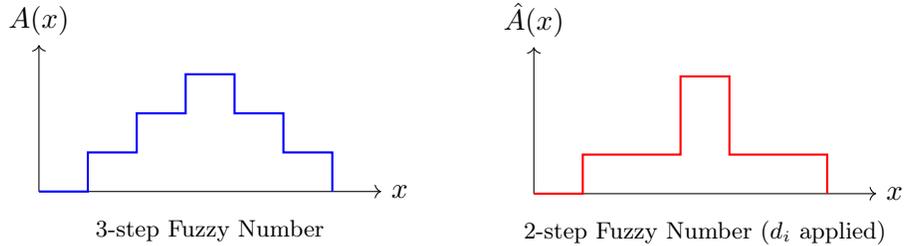
\begin{figure}[H]
    \centering
    \begin{minipage}{0.48\textwidth}
        \centering
        \begin{tikzpicture}[scale=1.3]
            \draw[->] (0,0) -- (3.5,0) node[right] {$x$};
            \draw[->] (0,0) -- (0,1.5) node[above] {$A(x)$};
            \draw[thick, blue] (0,0) -- (0.5,0) -- (0.5,0.4) -- (1,0.4) -- (1,0.8) -- (1.5,0.8) -- (1.5,1.2) -- (2,1.2) -- (2,0.8) -- (2.5,0.8) -- (2.5,0.4) -- (3,0.4) -- (3,0);
            \node at (1.75,-0.4) {\footnotesize 3-step Fuzzy Number};
        \end{tikzpicture}
    \end{minipage}
    \hfill
    \begin{minipage}{0.48\textwidth}
        \centering
        \begin{tikzpicture}[scale=1.3]
            \draw[->] (0,0) -- (3.5,0) node[right] {$x$};
            \draw[->] (0,0) -- (0,1.5) node[above] {$\hat{A}(x)$};
            \draw[thick, red] (0,0) -- (0.5,0) -- (0.5,0.4) -- (1.5,0.4) -- (1.5,1.2) -- (2,1.2) -- (2,0.4) -- (3,0.4) -- (3,0);
            \node at (1.75,-0.4) {\footnotesize 2-step Fuzzy Number ($d_i$ applied)};
        \end{tikzpicture}
    \end{minipage}
    \caption{Transformation of a 3-step fuzzy number into a 2-step fuzzy number using simplicial face maps.}
\end{figure}
\end{exa}

\begin{exa}
Consider a multi-expert decision-making scenario where two experts, $E_1$ and $E_2$, evaluate the same alternative. Expert $E_1$ provides a BUI granule $(x, c) = (0.7, 0.8)$, indicating a high evaluation with strong certainty. Expert $E_2$, who is more specialized, provides a CII granule $(y, [a^-, a^+]) = (0.6, [0.4, 0.9])$, representing an expected value with a specific acceptance interval.

To compare or aggregate these evaluations, we must embed them into a common simplicial space. Using the isomorphisms established in Propositions \ref{prop:BUI_Atan} and \ref{prop:L3isom}, we map these to the order polytopes $L_2$ and $L_3$:
\begin{itemize}
    \item For $E_1$: $(x,c) \mapsto I_{(0.7, 0.8)} = (0.56, 0.76) \in L_2$.
    \item For $E_2$: $(y, [a^-, a^+]) \mapsto (0.4, 0.6, 0.9) \in L_3$.
\end{itemize}

We can use the degeneracy maps $s_j: L_2 \to L_3$ to "lift" $E_1$'s preference into the higher-dimensional space of $E_2$. Applying $s_1$:
\begin{equation*}
    s_1(0.56, 0.76) = (0.56, 0.76, 0.76) \in L_3.
\end{equation*}
This identifies the BUI granule as a CII granule where the expected value coincides with the upper bound of the acceptance interval. Conversely, if we wish to simplify $E_2$'s information to the granularity of $E_1$, we apply the face map $d_1: L_3 \to L_2$:
\begin{equation*}
    d_1(0.4, 0.6, 0.9) = (0.4, 0.9) \in L_3.
\end{equation*}
This reduction removes the internal "expected value" $y=0.6$, preserving only the boundary uncertainty $[a^-, a^+]$, which can then be converted back to a BUI granule via the inverse mapping defined in Proposition \ref{prop:BUI_Atan}. 
\end{exa}

\begin{exa}
   This example relates to PLTSs and multi-granularity. Let us consider a group decision-making problem in which the preferences over the alternative set $X$ are given using PLTSs. However, the decision-makers show very different levels of expertise, and thus they need to use LTSs with different granularities: people with less experience will use $S_3=\{\text{low}, \text{medium}, \text{high}\}$, while those more experienced require more levels $S_5=\{\text{low}, \text{medium low},\allowbreak \text{medium}, \text{medium high}, \text{high}\}$. To address this issue, we need a way to embed the PLTSs on $S_3$ into the PLTSs on $S_5$. By our previous discussion, we know that these families are equivalent to $\Delta_4$ and $\Delta_6$ or, equivalently, to $L_3$ and $L_4$ or $\mathbb{A}_3$ and $\mathbb{A}_5$. Let us consider the bijection $\psi_n:\mathbb{A}_n\to L_n$. 
    \begin{gather*}        (p_1,p_2,p_3)\to\psi_3(p_1,p_2,p_3)=(p_1,p_1+p_2,p_1+p_2+p_3)\\
        \to s_0\circ\psi_3(p_1,p_2,p_3)=(p_1,p_1,p_1+p_2,p_1+p_2+p_3)\\\to s_2\circ s_0\circ\psi_3(p_1,p_2,p_3)=(p_1,p_1,p_1+p_2,p_1+p_2,p_1+p_2+p_3)\\
        \to \psi_5^{-1}s_2\circ s_0\circ\psi_3(p_1,p_2,p_3)=(p_1,0,p_2,0,p_3)
    \end{gather*}
    
    Thus, we can compose the corresponding bijection with the degeneracy maps and obtain the embedding $\psi_5^{-1}s_2\circ s_0\circ\psi_3:\mathbb{A}_3\to\mathbb{A}_5$ induced by the simplicial set structure.
\end{exa}

\section{Conclusion}\label{sec:conclusion}

In this work, we have provided a unifying geometric framework for the study of uncertain information and preference structures by identifying them with the simplicial structure of $n$-dimensional fuzzy sets. 

Throughout our study, we have revised and integrated many classical and modern preference structures from the literature, demonstrating that they are distinct semantic interpretations of the same topological objects. Specifically, we have shown the equivalence between 2-dimensional fuzzy sets and structures such as: Interval-valued and Atanassov intuitionistic fuzzy sets \cite{ATANASSOV1989343, CousoAtanasov},  Basic Uncertain Information (BUI) granules \cite{BUI}, Fuzzy rough sets \cite{pawlak1982rough} and shadowed sets \cite{Shadowed}, Grey sets \cite{YANG2012249} and vague sets \cite{vague}.

Furthermore, we extended this analysis to higher dimensions, connecting 3-dimensional fuzzy sets with Cognitive Interval Information (CII) granules \cite{HCYL}, Asymmetric Interval Numbers (AIN) \cite{Salabun}, and Picture Fuzzy Sets \cite{Cuong2014}.

A central contribution of this paper is the introduction of a new interpretable preference structure, namely $n$-dimensional Interval (type) cognitive uncertain information,  that can be constructed using Deck-of-Cards membership functions \cite{DoCMF}. This approach generalizes the revised structures by providing a flexible mechanism to represent complex membership degrees. By framing these as $n$-dimensional fuzzy sets, we leverage the geometric properties of the order polytope $L_n$, where the reduction in volume compared to the unit hypercube $[0,1]^n$ quantifies the information gain achieved by imposing monotonicity constraints.

Finally, we have established a simplicial structure for the set of multidimensional fuzzy sets $L_\infty$. By utilizing face maps $d_i$ and degeneracy maps $s_j$, we provide a formal way to increase or decrease the granularity of information, effectively unifying existing structures into a single simplicial set $L_\bullet$. This connection is not merely theoretical; as demonstrated in the examples in our final section, it allows for the seamless transformation between different levels of uncertainty, such as embedding preferences from lower-granularity scales into higher ones, providing a robust foundation for group decision-making under heterogeneous expertise.

\bibliographystyle{elsarticle-harv}

\end{document}